\documentclass[11pt,leqno]{article}
\usepackage[margin=1in]{geometry} 
\usepackage{amssymb,amsfonts,amsmath,bbm,mathrsfs,stmaryrd,mathtools}
\usepackage{xcolor}
\usepackage{url}

\usepackage{extarrows}

\usepackage[shortlabels]{enumitem}
\usepackage{tensor}

\usepackage{xr}
\usepackage[T1]{fontenc}

\usepackage[colorlinks,
linkcolor=black!75!red,
citecolor=blue,
pdftitle={},
pdfproducer={pdfLaTeX},
pdfpagemode=None,
bookmarksopen=true,
bookmarksnumbered=true,
backref=page]{hyperref}

\usepackage{tikz}
\usetikzlibrary{arrows,calc,decorations.pathreplacing,decorations.markings,decorations.shapes,intersections,shapes.geometric,through,fit,shapes.symbols,positioning,decorations.pathmorphing}

\usepackage{braket}

\usepackage[amsmath,thmmarks,hyperref]{ntheorem}
\usepackage{cleveref}

\newcommand\nthalias[1]{\AddToHook{env/#1/begin}{\crefalias{lemma}{#1}}}

\nthalias{definition}
\nthalias{example}
\nthalias{examples}
\nthalias{remark}
\nthalias{remarks}
\nthalias{convention}
\nthalias{notation}
\nthalias{construction}
\nthalias{theoremN}
\nthalias{propositionN}
\nthalias{corollaryN}
\nthalias{lemma}
\nthalias{proposition}
\nthalias{corollary}
\nthalias{theorem}
\nthalias{conjecture}
\nthalias{question}
\nthalias{assumption}

\creflabelformat{enumi}{#2#1#3}

\crefname{section}{Section}{Sections}
\crefformat{section}{#2Section~#1#3} 
\Crefformat{section}{#2Section~#1#3} 

\crefname{subsection}{\S}{\S\S}
\AtBeginDocument{%
  \crefformat{subsection}{#2\S#1#3}%
  \Crefformat{subsection}{#2\S#1#3}%
}

\crefname{subsubsection}{\S}{\S\S}
\AtBeginDocument{%
  \crefformat{subsubsection}{#2\S#1#3}%
  \Crefformat{subsubsection}{#2\S#1#3}%
}

\theoremstyle{plain}

\newtheorem{lemma}{Lemma}[section]
\newtheorem{proposition}[lemma]{Proposition}

\newtheorem{theorem}[lemma]{Theorem}

\theoremstyle{plain}
\theoremnumbering{Alph}

\theoremstyle{plain}
\theorembodyfont{\upshape}
\theoremsymbol{\ensuremath{\blacklozenge}}

\newtheorem{definition}[lemma]{Definition}
\newtheorem{example}[lemma]{Example}

\newtheorem{remark}[lemma]{Remark}
\newtheorem{remarks}[lemma]{Remarks}

\newtheorem{construction}[lemma]{Construction}

\crefname{definition}{definition}{definitions}
\crefformat{definition}{#2definition~#1#3} 
\Crefformat{definition}{#2Definition~#1#3} 

\crefname{ex}{example}{examples}
\crefformat{example}{#2example~#1#3} 
\Crefformat{example}{#2Example~#1#3} 

\crefname{exs}{example}{examples}
\crefformat{examples}{#2example~#1#3} 
\Crefformat{examples}{#2Example~#1#3} 

\crefname{remark}{remark}{remarks}
\crefformat{remark}{#2remark~#1#3} 
\Crefformat{remark}{#2Remark~#1#3} 

\crefname{remarks}{remark}{remarks}
\crefformat{remarks}{#2remark~#1#3} 
\Crefformat{remarks}{#2Remark~#1#3} 

\crefname{convention}{convention}{conventions}
\crefformat{convention}{#2convention~#1#3} 
\Crefformat{convention}{#2Convention~#1#3} 

\crefname{notation}{notation}{notations}
\crefformat{notation}{#2notation~#1#3} 
\Crefformat{notation}{#2Notation~#1#3} 

\crefname{table}{table}{tables}
\crefformat{table}{#2table~#1#3} 
\Crefformat{table}{#2Table~#1#3}

\crefname{lemma}{lemma}{lemmas}
\crefformat{lemma}{#2lemma~#1#3} 
\Crefformat{lemma}{#2Lemma~#1#3} 

\crefname{proposition}{proposition}{propositions}
\crefformat{proposition}{#2proposition~#1#3} 
\Crefformat{proposition}{#2Proposition~#1#3} 

\crefname{propositionN}{proposition}{propositions}
\crefformat{propositionN}{#2proposition~#1#3} 
\Crefformat{propositionN}{#2Proposition~#1#3} 

\crefname{corollary}{corollary}{corollaries}
\crefformat{corollary}{#2corollary~#1#3} 
\Crefformat{corollary}{#2Corollary~#1#3} 

\crefname{corollaryN}{corollary}{corollaries}
\crefformat{corollaryN}{#2corollary~#1#3} 
\Crefformat{corollaryN}{#2Corollary~#1#3} 

\crefname{theorem}{theorem}{theorems}
\crefformat{theorem}{#2theorem~#1#3} 
\Crefformat{theorem}{#2Theorem~#1#3} 

\crefname{theoremN}{theorem}{theorems}
\crefformat{theoremN}{#2theorem~#1#3} 
\Crefformat{theoremN}{#2Theorem~#1#3} 

\crefname{enumi}{}{}
\crefformat{enumi}{#2#1#3}
\Crefformat{enumi}{#2#1#3}

\crefname{assumption}{assumption}{Assumptions}
\crefformat{assumption}{#2assumption~#1#3} 
\Crefformat{assumption}{#2Assumption~#1#3} 

\crefname{construction}{construction}{Constructions}
\crefformat{construction}{#2construction~#1#3} 
\Crefformat{construction}{#2Construction~#1#3} 

\crefname{question}{question}{Questions}
\crefformat{question}{#2question~#1#3} 
\Crefformat{question}{#2Question~#1#3} 

\crefname{equation}{}{}
\crefformat{equation}{(#2#1#3)} 
\Crefformat{equation}{(#2#1#3)}

\numberwithin{equation}{section}

\theoremstyle{nonumberplain}
\theoremsymbol{\ensuremath{\blacksquare}}

\newtheorem{proof}{Proof}
\newcommand\pf[1]{\newtheorem{#1}{Proof of \Cref{#1}}}

\newcommand\bC{{\mathbb C}}

\newcommand\bG{{\mathbb G}}

\newcommand\bQ{{\mathbb Q}}
\newcommand\bR{{\mathbb R}}
\newcommand\bS{{\mathbb S}}
\newcommand\bT{{\mathbb T}}

\newcommand\bZ{{\mathbb Z}}

\newcommand\cA{{\mathcal A}}

\newcommand\cC{{\mathcal C}}

\newcommand\cE{{\mathcal E}}
\newcommand\cF{{\mathcal F}}

\newcommand\cL{{\mathcal L}}

\newcommand\cP{{\mathcal P}}

\newcommand\ol{\overline}

\DeclareMathOperator{\End}{\mathrm{End}}
\DeclareMathOperator{\Hom}{\mathrm{Hom}}

\DeclareMathOperator{\Aut}{\mathrm{Aut}}

\DeclareMathOperator{\Max}{\mathrm{Max}}

\DeclareMathOperator{\im}{\mathrm{im}}

\newcommand\spr[1]{\cite[\href{https://stacks.math.columbia.edu/tag/#1}{Tag {#1}}]{stacks-project}}
\newcommand{\qedhere}{\mbox{}\hfill\ensuremath{\blacksquare}}

\newcommand{\xrightarrowdbl}[2][]{%
  \xrightarrow[#1]{#2}\mathrel{\mkern-14mu}\rightarrow
}

\title{Algebra bundles, projective flatness and rationally-deformed tori}
\author{Alexandru Chirvasitu}

\begin{document}

\date{}

\newcommand{\Addresses}{{
  \bigskip
  \footnotesize

  \textsc{Department of Mathematics, University at Buffalo}
  \par\nopagebreak
  \textsc{Buffalo, NY 14260-2900, USA}  
  \par\nopagebreak
  \textit{E-mail address}: \texttt{achirvas@buffalo.edu}


}}

\maketitle

\begin{abstract}
  We show that isomorphism classes $[\mathcal{A}]$ of flat $q\times q$ matrix bundles $\mathcal{A}$ (or projectively flat rank-$q$ complex vector bundles $\mathcal{E}$) on a pro-torus $\mathbb{T}$ are in bijective correspondence with the \v{C}ech cohomology group $H^2(\mathbb{T},\mu_q:=\text{$q^{th}$ roots of unity})$ (respectively $H^2(\mathbb{T},\mathbb{Z})$) via the image of $[A]\in H^1(\mathbb{T},PGL(q,\mathcal{C}_{\mathbb{T}}))$ through $H^1(\mathbb{T},PGL(q,\mathcal{C}_{\mathbb{T}}))\xrightarrow{\quad}H^2(\mathbb{T},\mu(q,\mathcal{C}_{\mathbb{T}}))$ (respectively the first Chern class $c_1(\mathcal{E})$). This is in the spirit of Auslander-Szczarba's result identifying real flat bundles on the torus with their first two Stiefel-Whitney classes, and contrasts with classifying spaces $B\Gamma$ of compact Lie groups $\Gamma$ (as opposed to $\mathbb{T}^n\cong B\mathbb{Z}^n$), on which flat non-trivial vector bundles abound. The discussion both recovers the Disney-Elliott-Kumjian-Raeburn classification of rational non-commutative tori $\mathbb{T}^n_{\theta}$ with a different, bundle-theoretic proof, and sheds some light on the connection between topological invariants associated to $\mathbb{T}^2_{\theta}$, $\theta\in\mathbb{Q}$ by Rieffel and respectively H{\o}egh-Krohn-Skjelbred.
\end{abstract}


\vspace{.5cm}

\noindent{MSC 2020: 55R40; 55R37; 46M20; 16S35; 20J06; 46L85; 54B40; 55R25
  
}


\section*{Introduction}

The original motivation for the present investigation lies in work carried out in \cite{hks_erg,rief_canc} towards the classification of \emph{non-commutative tori} $\bT^2_{\theta}$, $\theta\in \bR$ defined \cite[Example 1.1.9]{khal_basic} as objects formally dual to the $C^*$-algebras
\begin{equation*}
  C(\bT^2_{\theta})
  :=
  \Braket{\text{unitaries }u,v\ :\ vu=e^{2\pi i \theta}uv}. 
\end{equation*}
For lowest-terms rational $\theta=\frac pq$ these turn out (\cite[Proposition 1.1.1]{khal_basic}, \cite[Proposition 12.2]{gbvf_ncg}) to be $q\times q$ matrix bundles $\cA\cong \cE\otimes \cE^*$ over $\bT^2$ (with $\cE$ being a rank-$q$ vector bundle), and the isomorphism problem for $C(\bT^2_{\theta})$, $\theta\in \bQ$ is resolved
\begin{itemize}[wide]
\item  in \cite[Theorems 3.9 and 3.12]{rief_canc} by relying on a \emph{twist} (an integer: \cite[p.299]{rief_canc}) attached to the rank-$q$ $\cE\xrightarrowdbl{}\bT^2$ and denoted below by $\mathrm{tw}(\cE)$;

\item and in \cite[Theorem 3.1]{hks_erg} by employing an invariant
  \begin{equation}\label{eq:omegaa}
    \omega(\cA)\in \bZ/q\subset \bS^1
  \end{equation}
  (\cite[Definition post Proposition 3.1]{hks_erg}).
\end{itemize}
The desire to clarify the relationship between twists and $\omega$ provided the initial impetus for the sequel. As it turns out,
\begin{itemize}[wide]
\item $\mathrm{tw}(\cE)$ is essentially (modulo sign/orientation choices and such irrelevancies) the \emph{first Chern class} \cite[\S 14.2]{ms_char} $c_1(\cE)\in H^2(\bT^2,\bZ)\cong \bZ$: \Cref{le:rief};

\item regarding \cite[Assertion 18.3.2]{hjjm_bdle} the isomorphism class of $\cA$ as a homotopy class $[f]$ of maps $\bT^2\to BPU(q)$, $\omega(\cA)$ is effectively (again, up to a sign: \Cref{pr:whatisomega}) the cohomology class
  \begin{equation*}
    \beta(\cA)
    :=
    [\gamma\circ f]
    \quad
    \text{\cite[\S 18.3.7]{hjjm_bdle}},
  \end{equation*}
  where
  \begin{equation*}
    BPU(q)
    \xrightarrow{\quad\gamma\quad}
    B^2\bZ/q
    \cong
    K(\bZ/q,2)
    \quad
    \left(
      \text{\emph{Eilenberg-MacLane space} \cite[Definition 9.6.1]{hjjm_bdle}}
    \right)    
  \end{equation*}
  is a portion of the classifying-space fiber sequence of \cite[\S 18.3.6]{hjjm_bdle} (indicating that it is perhaps not as ad-hoc as \cite[\S 3]{rief_canc} or \cite[\S 0, p.106]{MR962286} seem to suggest);

\item and finally (\Cref{pr:2constr}), $\omega(\cE\otimes \cE^*)$ is (essentially, give or take a sign again) the image of $\mathrm{tw}(\cE)$ under the map
  \begin{equation*}
    \bZ
    \cong
    \pi_1(U(q))
    \xrightarrow{\quad}
    \pi_1(PU(q))
    \cong
    \mu_q
    :=
    \left\{\text{$q^{th}$ roots of unity}\right\}.
  \end{equation*}
\end{itemize}

As the bundles $\cE$ of \cite[Proposition 1.1.1]{khal_basic} (and featuring in \cite[Notation 3.7]{rief_canc} as projective modules over $C(\bT^2)$) relevant to quantum-torus classification turn out to be \emph{projectively flat}\footnote{Not, however, flat, as claimed in \cite[Proposition 1.1.1]{khal_basic}; see also \cite[Remark 4.4(1)]{cp_cont-equiv_xv3}.} (so that the resulting matrix bundle $\cA$ is \emph{flat}; \Cref{se:app} recalls these and other assorted vocabulary), this leads naturally to an examination of how and to what extent the various topological invariants discussed above characterize such bundles. 

\Cref{th:protoriok} below aims at a complex analogue, for projectively flat vector bundles and flat matrix algebras, of the result \cite[Theorem 3.3]{as_vbtor} casting the first two \emph{Stiefel-Whitney} classes \cite[Definition 10.3.7]{hjjm_bdle} $w_i(\cE)$, $i=1,2$ as a complete invariant for flat, \emph{real} vector bundles over tori. 

At relatively little additional cost, one can work over \emph{pro-tori} \cite[Definitions 9.30]{hm5} instead: compact connected abelian groups. These are also precisely
\begin{itemize}[wide]
\item the \emph{cofiltered limits} \spr{04AY}
  \begin{equation}\label{eq:protor.cofilt}
    \bT\cong \varprojlim_i \left(\text{tori }\bT^{n_i}\cong \left(\bS^1\right)^{n_{i}}\right)
    ,\quad
    \bT^{n_j}
    \xrightarrowdbl[\quad\text{onto}\quad]{\quad\varphi_{ij}\quad}
    \bT^{n_i}
    ,\quad
    i\le j
  \end{equation}
  (by \cite[Corollary 2.43]{hm5} and the fact that compact, connected abelian Lie groups are tori \cite[Theorem 4.2.4]{de});
  
\item or alternatively \cite[Corollary 8.5]{hm5}, the \emph{Pontryagin duals} $\widehat{\Gamma}:=\Hom(\Gamma,\bS^1)$ of torsion-free abelian $\Gamma$.
\end{itemize}

In addition to the \emph{Chern classes} \cite[Definition 10.3.1]{hjjm_bdle} $c_k(\cE)$ attached to a vector bundle $\cE$ we also need cohomology classes associated to $q\times q$ matrix bundles $\cA$ (over compact Hausdorff $X$). To that end:
\begin{itemize}[wide]
\item Regard (the isomorphism class of) $\cA$ as an element of $H^1(X,PGL(q,\cC_X))$ per \cite[Assertion 19.6.2]{hjjm_bdle}, with $\cC_X$ denoting the sheaf of continuous complex-valued functions on $X$.
  
\item Then map said element to
  \begin{equation}\label{eq:muq}
    \beta(\cA)
    =
    \beta_q(\cA)
    \in
    H^2(X,\mu(q,\cC_X))
    \cong
    H^2(X,\mu_q)
    ,\quad
    \begin{aligned}
      \mu(q,\bullet)
      &:=
        \text{sheaf of $q^{th}$ roots of 1}\\
      \mu_q
      &:=
        \text{$q$-torsion of }\bS^1\\
      &=\text{center of $SU(q)$}\subset \bS^1
    \end{aligned}
  \end{equation}
  \begin{equation}\label{eq:conn.coh}
    H^1(X,PGL(q,\cC_X))\xrightarrow{\quad}H^2(X,\mu(q,\cC_X))
  \end{equation}
  in \cite[Remark 19.6.3]{hjjm_bdle} (or the $\delta'$ of \cite[pp.48-49]{groth_brau-1}) resulting from the \emph{long exact cohomology sequence} \cite[\S II.2.2]{bred_shf_2e_1997} attached to
  \begin{equation*}
    1\to
    \mu(q,\cC_X)
    \xrightarrow{\quad}
    SL(q,\cC_X)
    \xrightarrow{\quad}
    PGL(q,\cC_X)
    \to 1.
  \end{equation*}
\end{itemize}

All of this in hand, the result alluded to above is:

\begin{theorem}\label{th:protoriok}
  Let $\bT$ be a pro-torus and $q\in\bZ_{>0}$
  \begin{enumerate}[(1),wide]
  \item\label{item:projflc1} The map
    \begin{equation}\label{eq:projflc1}
      \left(\text{rank-$q$ vector bundle }\cE\right)
      \xmapsto{\quad}
      c_1(\cE)\in H^2(\bT,\bZ)
    \end{equation}
    restricts to a bijection on the set of isomorphism classes of projectively flat rank-$q$ vector bundles.

  \item\label{item:flbeta} Similarly, the map
    \begin{equation}\label{eq:flbeta}
      \left(\text{$q\times q$ matrix bundle }\cA\right)
      \xmapsto{\quad\text{\Cref{eq:muq}}\quad}
      \beta_q(\cA)\in H^2(\bT,\mu_q)
    \end{equation}
    restricts to a bijection on the set of isomorphism classes of flat matrix bundles.
  \end{enumerate}
\end{theorem}

\Cref{th:rat.n.tori} gives an application of the preceding material to the isomorphism problem for $C^*$-algebras of the form $C(\bT^n_{\theta})\otimes M_m$, where $C(\bT^n_{\theta})$ is the \emph{non-commutative $n$-torus algebra} denoted by $\overline{A}_{\theta}$ on \cite[p.193]{rief_case} and $\theta\in M_n(\bQ)$ is a skew-symmetric deformation parameter. 

\begin{theorem}\label{th:rat.n.tori}
  For $m,m',n,n'\in \bZ_{\ge 1}$ and skew-symmetric $\theta,\theta'$ respectively in $M_n(\bQ)$ and $M_{n'}(\bQ)$ the following conditions are equivalent.
  \begin{enumerate}[(a),wide]
  \item\label{item:th:rat.n.tori:tnm} There is a $C^*$ isomorphism
    \begin{equation*}
      C\left(\bT^n_{\theta}\right)\otimes M_m
      \quad\cong\quad
      C\left(\bT^{n'}_{\theta'}\right)\otimes M_{m'}.
    \end{equation*}

  \item\label{item:th:rat.n.tori:tn} We have
    \begin{equation*}
      n=n'
      ,\quad
      m=m'
      \quad\text{and}\quad
      C\left(\bT^n_{\theta}\right)
      \quad\cong\quad
      C\left(\bT^{n}_{\theta'}\right).
    \end{equation*}
    
  \item\label{item:th:rat.n.tori:tht} We have
    \begin{equation*}
      n=n'
      ,\quad
      m=m'
      \quad\text{and}\quad
      \theta'\in \left\{T\theta T^t\ :\ T\in GL(n,\bZ)\right\}+M_n(\bZ).
    \end{equation*}
  \end{enumerate}
\end{theorem}

The case $m=m'=1$ (a higher analogue of the torus branch of \cite[Theorem A(1)]{2508.04922v1}) thus recovers, with a somewhat different approach, a portion of \cite[Theorem, p.137]{dekr}. 


\subsection*{Acknowledgments}

I am grateful for valuable input from B. Badzioch, M. Khalkhali, H. Li and B. Passer.


\section{Chern-type topological invariants for algebra bundles}\label{se:app}

\emph{Bundles} (general \emph{fiber} or \emph{principal} \cite[\S 14.1]{td_alg-top}) are throughout to be understood in the sense of \cite[\S 3.1]{td_alg-top} or, say, \cite[Definitions 10.29]{hm5} (so are in particular locally trivial). A few brief reminders:
\begin{itemize}[wide]
\item A principal $\bG$-bundle is \emph{flat} \cite[\S I.2, Propositions 2.5 and 2.6]{kob_cplx} if the structure group $\bG$ can be \emph{reduced} (\cite[Definition 5.5.5]{hjjm_bdle}, \cite[post Example 14.1.15]{td_alg-top}) along
  \begin{equation*}
    \left(\bG,\ \text{discrete topology}\right)
    \xrightarrow{\quad}
    \left(\bG,\ \text{given topology}\right).
  \end{equation*}
  
\item A rank-$q$ vector bundle (or the attached \cite[Assertion 18.2.3]{hjjm_bdle} principal $GL(q,\bC)$- or $U(q)$-bundle) is \emph{projectively flat} \cite[\S I.2, post Proposition 2.6]{kob_cplx} if the associated \cite[Definition 5.3.1]{hjjm_bdle} principal bundle over
  \begin{equation*}
    PGL(q,\bC):=GL(q,\bC)/\bC^{\times}
    \quad\text{or}\quad
    PU(q):=U(q)/\left(\text{central }\bS^1\right)
  \end{equation*}
  is flat.
\end{itemize}
As customary, we write $B\bG$ for the \emph{classifying space} \cite[Definition 7.2.7]{hjjm_bdle} of a topological group $\bG$; in all cases of interest it will have the homotopy type of a CW-complex (e.g. because \cite[Theorem 5.1]{miln_univ-2} applies). 


In order to preserve compatibility with both \cite{hks_erg} and \cite{rief_canc} we will be working with rank-$q$ vector bundles, which we frequently assume \emph{Hermitian} \cite[\S 14.1]{ms_char} (so the structure group is $U(q)$). Per \cite[Remark 18.3.3]{hjjm_bdle}, this conflation of $BGL(q,\bC)$ with $BU(q)$ is harmless (as is that of $BPGL(q,\bC)$ with $BPU(q)$, for that matter):
\begin{equation*}
  U(q)\le GL(q,\bC)
  \quad\text{and}\quad
  PU(q)\le PGL(q,\bC)
\end{equation*}
are homotopy equivalences.


Note that the identification $H^2(\bT^2,\bZ)\cong \bZ$ requires an orientation on $\bT^2$ (i.e. \cite[preceding Proposition 3.3]{bt_forms} a nowhere-vanishing 2-form), so is not canonical. The approach to bundle classification adopted in \cite[\S 3]{rief_canc}, on the other hand, is different: the classifying invariant associated to $\cE$ there is an integer with no reference to orientation (though we will see one is implicit in the procedure). Paraphrasing, the procedure is as follows:

\begin{construction}\label{con:rieftw}
  Consider a rank-$q$ Hermitian vector bundle $\cE$ on the 2-torus $\bT^2$.

  \begin{enumerate}[(1),wide]

  \item\label{item:torquot} Identify $\bT^2$ with the quotient of $\bS^1\times\bR$ by the translation
    \begin{equation*}
      (s,t)\xmapsto{\quad}(s,t+1).
    \end{equation*}
    Equivalently, one can regard this as gluing $\bS^1\times I$, $I:=[0,1]$ by identifying the two boundary circles $\bS^1\times\{0\}$ and $\bS^1\times\{1\}$ in the obvious fashion.

  \item\label{item:trivesi} Trivialize $\cE$ over $\bS^1\times I$. This is always possible, given that
    \begin{itemize}[wide]
    \item bundles over $\bS^1\times I$ reduce, for classification purposes, to bundles over $\bS^1$ \cite[Theorems 11.4 and 11.5]{steen_fib};

    \item and bundles over $\bS^1$ are trivial in the present context because \cite[Corollary 18.6]{steen_fib} the structure group $U(q)$ is path-connected. 
    \end{itemize}

  \item The isomorphism class of $\cE$ will then be determined by the (homotopy class of the) map
    \begin{equation}\label{eq:s12uq}
      \bS^1\cong \bS^1\times\{0\}\cong \bS^1\times\{1\}\xrightarrow{\quad}U(q) 
    \end{equation}
    needed to identify the trivialized bundles on the two boundary circles upon gluing those circles back together.

  \item Because the segment
    \begin{equation*}
      \{1\}
      \cong
      \pi_1(SU(q))
      \xrightarrow{\quad}
      \pi_1(U(q))
      \xrightarrow{\ \pi_1(\det)\ }
      \bZ
      \cong
      \pi_1(\bS^1)
      \xrightarrow{\quad}
      \pi_0(SU(q))
      \cong
      \{1\}
    \end{equation*}
    of the long exact homotopy sequence (\cite[\S 17.3]{steen_fib}, \cite[Theorem 6.3.2]{td_alg-top}) attached to
    \begin{equation*}
      \{1\}
      \to
      SU(q)
      \xrightarrow{\quad}
      U(q)
      \xrightarrow{\ \det\ }
      \bS^1
      \to
      \{1\}
    \end{equation*}
    associates an integer to the map \Cref{eq:s12uq} {\it canonically}, we have our invariant: the {\it twist} of \cite[p.299]{rief_canc}.   
  \end{enumerate}
\end{construction}

\begin{remark}\label{re:autofact}
  Consider the matrix
  \begin{equation}\label{eq:rieff.n}
    N=N(s)=
    \begin{pmatrix}
      0 & 1 & 0 & \cdot & \cdot & 0\\
      0 & 0 & 1 & 0 &\cdot & 0 \\
      \cdot & \cdot & \cdot & \cdot & \cdot & \cdot \\
      0 & \cdot & \cdot & \cdot & \cdot & 1\\
      e(-as) & 0 & \cdot & \cdot & \cdot & 0
    \end{pmatrix}
    ,\quad
    e(\bullet)
    :=
    \exp(2\pi i \bullet)
  \end{equation}
  of \cite[text preceding Theorem 3.9]{rief_canc}. It allows us to recover the bundle $X(q,a)$ discussed there, of rank $q\in \bZ_{>0}$ and twist $-a\in \bZ$ \cite[Theorem 3.9]{rief_canc}, via the usual \cite[\S IV.7]{kob_cplx} {\it factor-of-automorphy} construction:
  \begin{equation*}
    \Gamma\cong \bZ^2\ni \gamma=(u,v)
    \xmapsto{\quad}
    \left(\bR^2\ni (s,t)=x\xmapsto{\quad N_{\gamma}\quad} N(s)^v\right)
  \end{equation*}
  is a factor of automorphy in the sense that it satisfies the {\it cocycle condition}
  \begin{equation*}
    N_{\gamma+\gamma'}(x) = N_{\gamma}(x+\gamma')\cdot N_{\gamma'}(x),\quad \forall \gamma,\gamma'\in \Gamma,\ x\in \bR^2.
  \end{equation*}
  One can substitute $x$ for $x+\gamma'$ on the right-hand side in this specific case, but the cocycle condition itself is that of \cite[\S IV.7, (7.2)]{kob_cplx} (or \cite[\S I.2]{mum_ab} for $q=1$, i.e. line bundles). The name (of the condition) is justified by the fact that
  \begin{equation*}
    \Gamma\ni \gamma
    \xmapsto{\quad}
    N_{\gamma}\in
    C(\bR^2,\ U(q)):=\mathrm{Cont}\left(\bR^2\xrightarrow{\quad}U(q)\right)    
  \end{equation*}
  is a non-abelian {\it 1-cocycle} for the action of $\Gamma$ on $C(\bR^2,\ U(q))$ by translation on the base. Conventions differ on what this means: to reconcile this picture with the notion of cocycle in \cite[\S 5.1]{ser_galcoh} one would have to consider
  \begin{equation*}
    \gamma\xmapsto{\quad}N_{\gamma}^{t}:=\text{transpose of }N_{\gamma}
  \end{equation*}
  instead (note the ordering difference between \cite[\S IV.7, (7.5)]{kob_cplx} and \cite[\S 5.1]{ser_galcoh} in defining cohomologous 1-cocycles). 

  One can now recover the (total space of the) bundle $X(q,a)$ from the 1-cocycle $(N_{\gamma})_{\gamma}$ as the quotient $\Gamma\backslash\bR^2\times \bC^q$ for the action
  \begin{equation*}
    \bR^2\times \bC^q\ni
    (x,v)
    \xmapsto{\quad\gamma\in\Gamma\quad}
    \left(x+\gamma,\ N_{\gamma}(x)v\right)
    \in \bR^2\times \bC^q.
  \end{equation*}
  The determinant\footnote{The claim on \cite[p.299]{rief_canc} that the $N$ of \Cref{eq:rieff.n} has determinant $(-1)^{q}e(-as)$ appears to be a typo: the sign is that of a permutation consisting of a single length-$q$ cycle, i.e. $(-1)^{q-1}$.}
  \begin{equation*}
    \left(\det N_{\gamma}\right)_{\gamma}
    =
    \left(\gamma\xrightarrow \det N_{\gamma}\right)
    =
    \left(\gamma=(u,v)\xmapsto{\ } (s,t)\xmapsto{\ }(-1)^{v(q-1)}e(-avs)\right)
    \in H^1(\Gamma,\bS^1)
  \end{equation*}
  is then a factor of automorphy defining the {\it determinant line bundle} $\bigwedge^q \cE$ \cite[\S IV.7, preceding (7.30)]{kob_cplx} (with `$\bigwedge^{\bullet}$' denoting exterior powers), so {\it its} resulting Chern class is the same as the original first Chern class $c_1(\cE)$ (e.g. \cite[Theorem 4.4.3 (IV)]{hirz_top}). These observations will be useful in the sequel. 
\end{remark}

Recall next that for tori
\begin{equation*}
  \bT^n\cong \bR^n/\left(\Gamma\cong \bZ^n\right)
\end{equation*}
we have {\it canonical} (\cite[Corollary 1.3.2]{bl_book}, \cite[\S I.1 (4)]{mum_ab}) identifications
\begin{equation}\label{eq:cohalt}
  H^p(\bT^n,\bZ)\cong \mathrm{Alt}^p(\Gamma,\bZ)\cong \bigwedge^p\mathrm{Hom}(\Gamma,\bZ)\cong \mathrm{Hom}(\bigwedge^p\Gamma,\bZ)
\end{equation}
where `$\mathrm{Alt}^{\bullet}$' denotes alternating (or skew-symmetric) multilinear forms and `$\bigwedge^{\bullet}$' again denotes exterior powers (the fact that \cite{bl_book,mum_ab} are concerned with {\it complex} tori makes no difference here). 

\begin{remark}\label{re:howcup}
  Even the canonical identifications \Cref{eq:cohalt} require some discussion of conventions (though these turn out not to make a difference to the resulting isomorphisms). \cite[Corollary 1.3.2]{bl_book} and \cite[\S I.1 (4)]{mum_ab} both proceed by first proving \Cref{eq:cohalt} for $p=1$ and then extending to the general case by taking {\it cup products} on both sides: on singular cohomology on the one hand, and multilinear forms on the other.
 
  There are two conventions in the literature for how such cup products are to be constructed (see for instance the discussion in \cite[Appendix C, Note on signs]{ms_char}). Focusing on singular cohomology, one might
  \begin{enumerate}[(a),wide]
  \item multiply two cocycles $f_p$ and $g_q$, of degrees $p$ and $q$ respectively, straightforwardly (e.g. \cite[\S 5.6]{spa_at}, \cite[\S 3.2]{hatcher}):
    \begin{equation*}
      (f_p g_q)(\text{simplex }\sigma)
      :=
      f_p(\text{front $p$-face of }\sigma)
      g_q(\text{back $q$-face of }\sigma)
    \end{equation*}
  \item or multiply them observing the {\it sign rule} \cite[\S 11.7.1]{td_alg-top} of always scaling by $(-1)^{st}$ whenever homogeneous symbols of respective degrees $s$ and $t$ are interchanged (\cite[\S A.5]{ms_char}, \cite[\S 17.6]{td_alg-top}, \cite[Definition VII.8.1]{dold_at-2e}):
    \begin{equation*}
      (f_p g_q)(\text{simplex }\sigma)
      :=
      (-1)^{pq}
      f_p(\text{front $p$-face of }\sigma)
      g_q(\text{back $q$-face of }\sigma).
    \end{equation*}
  \end{enumerate}
  The same dichotomy applies to exterior products of alternating forms, but so long as one is consistent about observing the sign rule (either on both sides of \Cref{eq:cohalt} or on neither side) \Cref{eq:cohalt} will be canonical. 

  We {\it will} follow the sign rule whenever the issue arises. 
\end{remark}

As for orientations:

\begin{definition}\label{def:rief-orient}
  In the context of \Cref{con:rieftw}, the {\it standard (or associated) orientation} is the one induced by the nowhere-vanishing form $ds\wedge dt$.
\end{definition}

\begin{lemma}\label{le:rief}
  Let $\cE$ be a rank-$q$ vector bundle $\cE$ on $\bT^2$ and write
  \begin{itemize}[wide]
  \item $\mathrm{tw}(\cE)\in \bZ$ for the twist of $\cE$ as defined on \cite[p.299]{rief_canc} or via \Cref{con:rieftw};
  \item $[\bT^2]\in H_2(\bT^2,\bZ)$ for the top homology class corresponding to the standard orientation of \Cref{def:rief-orient};
  \item and $c_1(\cE)\in H^2(\bT^2,\bZ)$ as usual, for the first Chern class of $\cE$.
  \end{itemize}
  We then have
  \begin{equation}\label{eq:twminus}
    \mathrm{tw}(\cE) = -c_1(\cE)[\bT^2]\in \bZ. 
  \end{equation}
\end{lemma}
\begin{proof}
  it will be enough to consider the bundles $X(q,a)$ of \cite[\S 3]{rief_canc}, as these are cover all isomorphism classes of rank-$q$ bundles on $\bT^2$ (as follows from \cite[discussion preceding Theorem 3.9]{rief_canc}).

  We saw in \Cref{re:autofact} that $c_1(\cE)=c_1\left(\bigwedge^q\cE\right)$ is computable from the factor of automorphy defined uniquely by 
  \begin{equation*}
    (1,0)\xmapsto{\quad} 1
    ,\quad
    (0,1)\xmapsto{\quad} (-1)^{q-1}e(-as).
  \end{equation*}
  In general, for a factor of automorphy given by $\gamma\xmapsto{} e(f_{\gamma})$, the corresponding Chern class, regarded as an alternating 2-form on $\Gamma$, is (\cite[\S I.2, Proposition]{mum_ab} or \cite[Theorem 2.1.2]{bl_book})
  \begin{equation*}
    (\gamma_1,\ \gamma_2)
    \xmapsto{\quad}
    \left(f_{\gamma_2}(x+\gamma_1) - f_{\gamma_2}(x)\right)
    -
    \left(f_{\gamma_1}(x+\gamma_2) - f_{\gamma_1}(x)\right)
  \end{equation*}
  (for any $x\in \bR^2$). 

  In our case this would be the skew-symmetric bilinear form on $\bZ^2$ defined uniquely by
  \begin{equation*}
    \left((1,0),\ (0,1)\right)
    \xmapsto{\quad}
    -a
    =
    \text{Rieffel's twist of \cite[p.299]{rief_canc}}.
  \end{equation*}
  Because of our sign convention though (\Cref{re:howcup}), for a skew-symmetric bilinear form $E$ on $\Gamma$ representing the same cohomology class as a differential form $\omega$ on the torus, $E\left((1,0),(0,1)\right)$ will be {\it minus} the scalar multiple $k$ in
  \begin{equation*}
    \omega\text{ cohomologous with }k\cdot ds\wedge dt.
  \end{equation*}
\end{proof}

\begin{remarks}\label{res:when+}
  \begin{enumerate}[(1),wide]
  \item \Cref{eq:twminus} and the proof of \Cref{le:rief} are compatible with the standard positivity conventions: focusing on the lattice
    \begin{equation*}
      \Gamma:=\bZ\oplus \bZ i\subset \bC^2\cong \bR^2,
    \end{equation*}
    the line bundle with Chern class $E$ (a skew-symmetric bilinear form on $\Gamma$) can have non-zero holomorphic sections only if 
    \begin{equation*}
      H(x,y):=E(ix,y)+iE(x,y)
    \end{equation*}
    is positive-definite \cite[\S I.3, preceding Proposition]{mum_ab}. This translates to
    \begin{equation*}
      E((0,1),\ (1,0)) = E(i,1)>0,
    \end{equation*}
    which in turn corresponds to the positivity of the degree $c_1[\bT^2]$ (since non-trivial line bundles on elliptic curves have non-zero (holomorphic) sections if and only if they have {\it positive} degree \cite[Lemma 4.1 and \S 2.3]{daly-r2}). For that reason, the degree and $E((1,0),\ (0,1))$ will have {\it differing} signs.

  \item\label{item:needorient} As \Cref{def:rief-orient} and \Cref{eq:twminus} both suggest, the twist does depend on the orientation. Consider, for instance, the case of a line bundle $X(1,a)$, as in \cite[Notation 3.7]{rief_canc}. Its sections are the continuous functions
    \begin{equation*}
      h\in C(\bR^2,\bC),\quad h(s+1,t)=h,\quad h(s,t+1)=e(-as)h(s,t).
    \end{equation*}
    The functions
    \begin{equation*}
      (s,t)\xmapsto{\quad}h(s,t)\cdot e(ast)
    \end{equation*}
    are then sections of an isomorphic line bundle, with the roles of $s$ and $t$ interchanged (hence the change of orientation) and $a$ replaced with $-a$.

    In short: the twist can be regarded as a line-bundle invariant only after we have fixed an orientation.
  \end{enumerate}
\end{remarks}

We will also connect the preceding discussion to the invariant \Cref{eq:omegaa} attached to $q\times q$ matrix bundles $\cA$ over $\bT^2$ (equivalently \cite[Assertion 18.2.4]{hjjm_bdle}: $PU(q)$-bundles) on \cite[\S 3, pp.5-6]{hks_erg}. Purposely rephrased so as to make the analogy to \Cref{con:rieftw} plain, the procedure reads as follows.

\begin{construction}\label{con:hksw}
  Consider a $q\times q$ matrix-algebra bundle $\cA$ on the 2-torus $\bT^2$.

  \begin{enumerate}[(1),wide]
  \item As in \Cref{con:rieftw}\Cref{item:torquot}, identify $\bT^2$ with a quotient of $\bS^1\times I$, $I:=[0,1]$ by identifying the two boundary circles.

  \item Trivialize $\cA$ over $\bS^1\times I$; this is possible for the same reason it was in \Cref{con:rieftw}\Cref{item:trivesi}: the structure group $PU(q)$ is path-connected.

  \item The isomorphism class of $\cA$ will then be determined by the (homotopy class of the) map
    \begin{equation}\label{eq:s12puq}
      \bS^1\cong \bS^1\times\{0\}\cong \bS^1\times\{1\}\xrightarrow{\quad}PU(q) 
    \end{equation}
    needed to identify the trivialized bundles on the two boundary circles upon gluing those circles back together.

  \item The segment
    \begin{equation*}
      \{1\}
      \cong
      \pi_1(SU(q))
      \xrightarrow{\quad}
      \pi_1(PU(q))
      \xrightarrow[\cong]{\quad}
      \mu_q
      \cong
      \pi_0(\mu_q)
      \xrightarrow{\quad}
      \pi_0(SU(q))
      \cong
      \{1\}
    \end{equation*}
    of the long exact homotopy sequence (\cite[\S 17.3]{steen_fib}, \cite[Theorem 6.3.2]{td_alg-top}) attached to
    \begin{equation*}
      \{1\}
      \to
      \mu_q
      \xrightarrow{\quad}
      SU(q)
      \xrightarrow{\quad}
      PU(q)
      \to
      \{1\}
    \end{equation*}
    associates an element of the $\mu_q$ of \Cref{eq:muq} canonically to the map \Cref{eq:s12puq}: $\omega(\cA)$, by definition (\cite[immediately preceding Lemma 3.3]{hks_erg}).
  \end{enumerate}
  Note that \cite[\S 3]{hks_erg} refers to $\mu_q$ as $\bZ_q$, suggesting $\bZ/q$, but the group being used is in fact \Cref{eq:muq}: no specific primitive root of unity is singled out, so no isomorphism $\mu_q\cong \bZ/q$ is chosen (see also \Cref{re:munclass}). 
\end{construction}

\begin{remark}\label{re:diff.orient}
  Once more an orientation is implicit in \Cref{con:hksw}, as explained in \Cref{res:when+}\Cref{item:needorient}. For that matter, the present discussion alters the original definition of $\omega(\bullet)$ in just this respect, in order to implement compatibility with \Cref{con:rieftw}:
  \begin{itemize}[wide]
  \item \cite[pp.5-6]{hks_erg} trivialize the bundle on $I\times \bS^1$ rather than $\bS^1\times I$;

  \item and then take for $\omega$ the {\it inverse} of the resulting element in $\pi_1(PU(q))\cong \mu_q$. 
  \end{itemize}
  As \Cref{res:when+}\Cref{item:needorient} indicates, the coordinate role-reversal precisely accounts for the sign difference/inversion. 
\end{remark}

A comparison of \Cref{con:rieftw} and \Cref{con:hksw} immediately proves

\begin{proposition}\label{pr:2constr}
  Fix an orientation on $\bT^2$. For a rank-$q$ vector bundle $\cE$ on $\bT^2$ the invariant
  \begin{equation*}
    \omega(\cE\otimes \cE^*)\in \mu_q\cong \pi_1(PU(q))
  \end{equation*}
  is the image of 
  \begin{equation*}
    \mathrm{tw}(\cE)\in \bZ\cong \pi_1(U(q))
  \end{equation*}
  through the map $\pi_1(U(q))\to \pi_1(PU(q))$ induced by the quotient $U(q)\to PU(q)$.  \qedhere
\end{proposition}


\begin{remark}\label{re:munclass}
  As the discussion in \Cref{con:hksw} anticipates, the 2-class
  \begin{equation*}
    \beta(\cA)=\beta_q(\cA)\in H^2(X,\bZ/q)
  \end{equation*}
  of a $q\times q$ matrix bundle $\cA$ discussed in \cite[Remark 1.6(4)]{2508.04922v1} is more appropriately thought of as lying in $H^2(X,\mu_q)$ instead (with $\mu_q$ as in \Cref{eq:muq}). The map associating $\beta$ to the bundle $\cA$ can then be identified with \Cref{eq:conn.coh} and in this interpretation, the cohomological counterpart to
  \begin{equation*}
    \cE\xmapsto{\quad}\cA:=\cE nd(\cE)\cong \cE\otimes \cE^*,
  \end{equation*}
  mapping the Chern class $c_1(\cE)$ to $\beta(\cA)$, is
  \begin{equation}\label{eq:z2muq}
    H^2(X,\bZ)\cong H^1(X,\cC_X^{*})\xrightarrow{\quad} H^2(X,\mu_q),
  \end{equation}
  where
  \begin{itemize}[wide]
  \item $\cC_X^*$ is the sheaf of nowhere-vanishing continuous complex-valued functions on $X$;
  \item the first isomorphism is the usual \cite[Remark 10.1.4]{hjjm_bdle} identification of the {\it Picard group} with $H^2(-,\bZ)$ via the first Chern class (also \cite[p.49, bottom]{groth_brau-1} for $i=1$);
  \item and the rightward map is part of the long exact cohomology sequence attached to
    \begin{equation*}
      1\xrightarrow{} \mu(q,\cC_X)
      \xrightarrow{\quad}
      \cC_X^*
      \xrightarrow{\quad\text{$q^{th}$ power}\quad}
      \cC_X^*
      \xrightarrow{}
      1
    \end{equation*}
    (labeled $\braket{n}$ in \cite[Remark 19.6.4]{hjjm_bdle}). 
  \end{itemize}
  This version of the map $c_1(\cE)\xmapsto{} \beta(\cE\otimes \cE^*)$ discussed in \cite[Remark 1.6(3)]{2508.04922v1} does not require a choice of generator for $\mu_q$: that choice implements an isomorphism $\mu_q\cong \bZ/q$, which the present discussion obviates.
\end{remark}

\Cref{le:rief} and \Cref{pr:2constr} now combine to identify the invariant $\omega$.

\begin{proposition}\label{pr:whatisomega}
  Let $\cA\xrightarrowdbl{} \bT^2$ be a $q\times q$ matrix bundle on $\bT^2$ and write
  \begin{itemize}[wide]
  \item $\omega(\cA)\in \mu_q$ for the invariant defined on \cite[pp.5-6]{hks_erg} or via \Cref{con:hksw};
  \item $[\bT^2]\in H_2(\bT^2,\bZ)$ for the top homology class corresponding to the standard orientation of \Cref{def:rief-orient};
  \item and $\beta(\cA)\in H^2(\bT^2,\mu_q)$ for the 2-class of the matrix bundle $\cA$ of \Cref{eq:muq} and/or \cite[\S 18.3.7]{hjjm_bdle}.
  \end{itemize}
  We then have
  \begin{equation*}
    \omega(\cA) = \beta(\cA)[\bT^2]^{-1}\in \mu_q. 
  \end{equation*}
\end{proposition}


\begin{proof}
  \Cref{le:rief} and \Cref{pr:2constr} (via \Cref{re:munclass}, concerning the relationship between Chern and $\beta$ classes) confirm this for $\cA\cong \cE\otimes \cE^*$. The cohomology $H^3(\bT^n,\bZ)\cong\bZ^{\oplus \binom{n}{3}}$ being torsion-free, all matrix bundles on tori are of the form $\cE\otimes \cE^*$ \cite[Remark 18.3.8]{hjjm_bdle}.
\end{proof}

\begin{remark}\label{re:c1.beta}
  The foregoing discussion takes it for granted that \Cref{eq:z2muq} indeed maps $c_1(\cE)$ to the class $\beta(\cE\otimes \cE^*)$ of \Cref{eq:muq}. This is the gist of \cite[Remark 1.6(3)]{2508.04922v1} and also follows (casting the discussion in classifying-space rather than sheaf language) from a diagram chase around (the commutative)
  \begin{equation*}
    \begin{tikzpicture}[>=stealth,auto,baseline=(current  bounding  box.center)]
      \path[anchor=base] 
      (0,0) node (l) {$BSU(q)$}
      +(4,.5) node (u) {$BU(q)$}
      +(4,-.5) node (d) {$BPU(q)$}
      +(8,0) node (r) {$B^2\mu_q$}
      +(0,1) node (ul) {$B\bS^1 \times BSU(q)$}
      +(6,1) node (ur) {$B\bS^1$}
      +(3,1.5) node (uu) {$B\bS^1$}
      +(1,2.2) node (uul) {$B\mu_q$}
      ;
      \draw[densely dotted,->] (l) to[bend left=6] node[pos=.5,auto] {$\scriptstyle $} (u);
      \draw[dashed,->] (u) to[bend left=6] node[pos=.5,auto] {$\scriptstyle $} (r);
      \draw[->] (l) to[bend right=6] node[pos=.5,auto,swap] {$\scriptstyle $} (d);
      \draw[->] (d) to[bend right=6] node[pos=.5,auto,swap] {$\scriptstyle $} (r);
      \draw[densely dotted,->] (u) to[bend left=6] node[pos=.5,auto] {$\scriptstyle B\det$} (ur);
      \draw[->] (ul) to[bend left=0] node[pos=.5,auto] {$\scriptstyle$} (l);
      \draw[dashed,->] (ul) to[bend left=6] node[pos=.5,auto] {$\scriptstyle$} (u);
      \draw[->] (ul) to[bend left=6] node[pos=.5,auto] {$\scriptstyle$} (uu);
      \draw[decorate,decoration={coil,aspect=0},->] (uu) to[bend left=16] node[pos=.5,auto] {$\scriptstyle B\left(-\right)^q$} (ur);
      \draw[decorate,decoration={coil,aspect=0},->] (ur) to[bend left=6] node[pos=.5,auto] {$\scriptstyle$} (r);
      \draw[->] (u) to[bend left=0] node[pos=.5,auto] {$\scriptstyle $} (d);
      \draw[decorate,decoration={coil,aspect=0},->] (uul) to[bend right=6] node[pos=.5,auto] {$\scriptstyle $} (uu);
    \end{tikzpicture}
  \end{equation*}
  where
  \begin{itemize}[wide]
  \item the left-hand $B\bS^1\to BU(q)$ \emph{deloops} \cite[post Theorem 8.22]{dk_at_2001} the central inclusion $\bS^1\le U(q)$;

  \item and the bottom (solid), dashed, dotted and squiggly paths are all \emph{fibration sequences} in the sense of \cite[\S 6.11]{dk_at_2001}. 
  \end{itemize}
\end{remark}

Turning to the higher-torus/pro-torus setup outlined in the Introduction, recall that flat {\it real} vector bundles on $\bT^n$ are
\begin{itemize}[wide]
\item sums of (real) line bundles \cite[Theorem 3.2]{as_vbtor};
\item uniquely determined \cite[Theorem 3.3]{as_vbtor} up to isomorphism by their first two {\it Stiefel-Whitney classes} \cite[Definition 10.3.7]{hjjm_bdle} (living in $H^i(-,\bZ/2)$, $i=1,2$).
\end{itemize}

The picture is even simpler for {\it complex} vector bundles, essentially because $GL(n,\bC)$ is connected (whereas $GL(n,\bR)$ is not): with only straightforward modifications, the proof of \cite[Theorem 3.2]{as_vbtor} also yields

\begin{lemma}\label{le:flatistriv}
  A flat complex vector bundle on a torus is trivial.  \qedhere
\end{lemma}

\begin{remark}  \label{re:justbdle}
  As has become customary by now in the present work, \Cref{le:flatistriv} refers to flat bundles being trivial as {\it just} bundles (not as {\it flat} bundles). Or: their structure group restricts along the continuous map
  \begin{equation*}
    \left(U(q),\ \text{discrete topology}\right)
    \xrightarrow{\quad}
    \left(U(q),\ \text{usual topology}\right),
  \end{equation*}
  but are trivial as $U(q)$-bundles for the {\it latter} (weaker) topology.
  
  This will be common practice henceforth: when speaking of isomorphic (projectively) flat bundles, we always mean (unless specified otherwise) isomorphisms as {\it plain} bundles. 
\end{remark}

\Cref{th:protoriok}\Cref{item:projflc1} of course strengthens \Cref{le:flatistriv}. Before turning to that proof, recalling that the torus $\bT^n$ is the classifying space $B\bZ^n$ of \cite[Definition 7.2.7]{hjjm_bdle}, note the contrast to classifying spaces of {\it finite} groups:

\begin{proposition}\label{pr:bgfing}
  \begin{enumerate}[(1)]
  \item\label{item:cpctlie} For a compact Lie group $\Gamma$ the complex rank-$q$ vector bundle on $B\Gamma$ corresponding to a non-trivial unitary representation $\rho:\Gamma\to U(q)$ is non-trivial.

  \item\label{item:hencefin} In particular, if $\Gamma$ is finite, every non-trivial $\rho$ gives rise to a non-trivial flat vector bundle on $B\Gamma$.
  \end{enumerate}
\end{proposition}
\begin{proof}

  \Cref{item:hencefin} is of course just \Cref{item:cpctlie} applied to finite groups, so it will be enough to prove the latter claim.

  Per \cite[Corollary 1.10]{jo_vb}, the elements of the {\it representation ring} \cite[Definition 13.5.1]{hus_fib} $R(\Gamma)$ annihilated by the map into the Grothendieck ring of vector bundles are precisely those whose characters vanish on all connected components of $\Gamma$ of prime-power order in
  \begin{equation*}
    \pi_0(\Gamma)\cong \Gamma/\Gamma_0,\quad \Gamma_0:=\text{connected component containing }1.
  \end{equation*}
  For elements of the form
  \begin{equation*}
    \rho-\dim(\rho)=\rho-q\in R(G),\quad \Gamma\xrightarrow{\ \rho\ }U(q)\quad\text{as in the statement}
  \end{equation*}
  such vanishing means precisely that $\rho$ is trivial: the character $\chi_{\rho}$ has as its values sums of $q$ modulus-1 complex numbers so takes the value $q$ only on $\gamma\in\ker~\rho$, and {\it every} connected component of $\Gamma$ is a product of connected components of prime-power orders.
\end{proof}


In addressing \Cref{th:protoriok}, recall \Cref{re:justbdle} regarding the meaning of the term `isomorphic'. 

\pf{th:protoriok}
\begin{th:protoriok}
  For pro-tori \Cref{eq:protor.cofilt} $H^*(\bT,\bullet)$ is understood as \emph{\v{C}ech} \cite[\S I.7, pp.27-29]{bred_shf_2e_1997} (equivalently \cite[Corollary III.4.12]{bred_shf_2e_1997}, sheaf) cohomology
  \begin{equation}\label{eq:coh.colim}
    H^*(\bT,\bullet)
    \quad
    \overset{\text{\cite[Theorem II.14.4]{bred_shf_2e_1997}}}{\cong}
    \quad
    \varinjlim_i H^*(\bT^{n_i},\bullet).
  \end{equation}
  There are several stages to the argument.  
  \begin{enumerate}[(I),wide]

  \item \textbf{ Reduction to tori.} This is a simple matter of unwinding the appropriate notion of continuity for the functors of interest in the present context. Write (following \cite[\S 4.10]{hus_fib}, say) $k_{\bG}$ for the functor assigning to $X$ the set of isomorphism classes of \emph{numerable} \cite[Definitions 4.9.1 and 4.9.2]{hus_fib} principal ${\bG}$-bundles thereon. Let ${\bG}$ range over
    \begin{enumerate}[(a),wide]
    \item $U(q)$ or $PU(q)$ with their usual topologies;

    \item\label{item:th:protoriok:pf.discpu} discrete $PU(q)$;

    \item\label{item:th:protoriok:pf.12discu} or the topological-group \emph{pullback} \cite[Definition 11.8]{ahs}
      \begin{equation*}
        \begin{tikzpicture}[>=stealth,auto,baseline=(current  bounding  box.center)]
          \path[anchor=base] 
          (0,0) node (l) {$\left(U(q),\text{standard topology}\right)$}
          +(4,.5) node (u) {$\bullet$}
          +(4,-.5) node (d) {$\left(PU(q),\text{standard}\right)$}
          +(8,0) node (r) {$\left(PU(q),\text{discrete}\right)$}
          ;
          \draw[<-] (l) to[bend left=6] node[pos=.5,auto] {$\scriptstyle $} (u);
          \draw[->] (u) to[bend left=6] node[pos=.5,auto] {$\scriptstyle $} (r);
          \draw[->] (l) to[bend right=6] node[pos=.5,auto,swap] {$\scriptstyle $} (d);
          \draw[<-] (d) to[bend right=6] node[pos=.5,auto,swap] {$\scriptstyle $} (r);
        \end{tikzpicture}
      \end{equation*}
    \end{enumerate}
    so that (Hermitian) vector, matrix-algebra, projectively-flat vector or flat matrix-algebra bundles respectively are all handled simultaneously. Observe that 
    \allowdisplaybreaks
    \begin{align*}
      k_{{\bG}}(\bT)
      &\cong
        \varinjlim_i k_{{\bG}}\left(\bT^{n_i}\right)
        \quad\text{by \cite[Theorem 4 and Remark 4.11]{MR234450}}\\
      &\xrightarrow[\quad\cong\quad]{\quad c_1\text{ or }\beta_q\quad}
        \varinjlim_i H^2(\bT^{n_i},\bZ\text{ or }\mu_q)
        \quad\left(\text{assuming the result for tori}\right)\\
      &\cong
        H^2(\bT,\bZ\text{ or }\mu_q)
        \quad\text{\Cref{eq:coh.colim}}.
    \end{align*}
    We will thus henceforth assume $\bT\cong \bT^n$. 
    
  \item \textbf{ Surjectivity.} For \Cref{item:projflc1} we can simply take $\cE$ to be a sum of a line bundle with prescribed Chern class \cite[Remark 10.1.4]{hjjm_bdle} (automatically projectively flat) and a trivial rank-$(q-1)$-bundle. As for \Cref{eq:flbeta}, its surjectivity follows from that of \Cref{eq:projflc1} via \Cref{eq:z2muq} by considering bundles of the form $\cA\cong \cE\otimes \cE^*$.

  \item \textbf{ Injectivity: \Cref{item:flbeta}} It will be convenient to work directly with {\it principal} bundles, in the present case over $PU(q)$. Consider two such, $\cP_i$, $i=1,2$, obtained respectively as quotients by the actions
    \begin{equation*}
      \bR^n\times PU(q)\ni
      (x,a)
      \xmapsto{\quad\gamma\in\Gamma:=\bZ^n\quad}
      (x+\gamma,\pi_i(\gamma)^{-1}a)
      \in \bR^n\times PU(q)
    \end{equation*}
    for representations
    \begin{equation*}
      \Gamma\xrightarrow{\quad\pi_i\quad}PU(q),\quad i=1,2.
    \end{equation*}
    Lifting the $\pi_i$ as maps (not {\it morphisms}, generally)
    \begin{equation*}
      \Gamma\xrightarrow{\quad\widehat{\pi}_i\quad}SU(q),\quad i=1,2.      
    \end{equation*}
    along $SU(q)\to PU(q)$ will produce the corresponding $\beta$ invariants as the 2-cohomology classes associated to the 2-cocycles (\cite[(1.1)]{slw_fact}, say)
    \begin{equation*}
      \Gamma^2\ni (\gamma,\gamma')
      \xmapsto{\quad z_i\quad}
      \widehat{\pi}_i(\gamma)\cdot \widehat{\pi}_i(\gamma')\cdot \widehat{\pi}_i(\gamma+\gamma')^{-1}
      ,\quad i=1,2.
    \end{equation*}
    We are assuming that $z_i$ are cohomologous, i.e.
    \begin{equation*}
      z_1(\gamma,\gamma')\cdot z_2(\gamma,\gamma')^{-1} = f(\gamma)\cdot f(\gamma')\cdot f(\gamma+\gamma')^{-1}
      ,\quad \forall \gamma,\gamma'\in\Gamma
    \end{equation*}
    for some
    \begin{equation*}
      \Gamma
      \xrightarrow{\quad f\quad}
      \text{central }\mu_q\subset SU(q).
    \end{equation*}
    Replacing $\widehat{\pi}_2$ with $f\cdot \widehat{\pi}_2$ allows us to assume the cocycles $z_i$ are in fact {\it equal} to a common unadorned $z$, as we henceforth will.

    Consider, now, the kernel $H\le \Gamma$ of the skew-symmetric bicharacter
    \begin{equation*}
      \Gamma^2\ni (\gamma,\gamma')
      \xmapsto{\quad\chi\quad}
      z(\gamma,\gamma')\cdot z(\gamma',\gamma)^{-1}\in \bS^1\subset U(q)
    \end{equation*}
    (that uniquely determines and corresponds to the cohomology class of $z$ \cite[Proposition 3.2]{opt_erg}):
    \begin{equation*}
      H:=\{h\in\Gamma\ |\ \chi(h,\gamma)=1,\ \forall \gamma\in\Gamma\}.
    \end{equation*}
    $\widehat{\pi}_i$ can then be assumed to be genuine representations on $H$. Because the latter is free abelian, the argument proving \Cref{le:flatistriv} (and \cite[Theorem 3.2]{as_vbtor}) further allows us to modify $\widehat{\pi}_i$ so as to preserve the isomorphism classes of the bundles while also ensuring that the restrictions $\widehat{\pi}_i|_{H}$ are trivial.

    We have now reduced the discussion to two projective representations of the quotient $\Gamma':=\Gamma/H$ which give rise to equal 2-cocycles on that group. Because (the bicharacter induced by) $\chi$ on $\Gamma'$ is by construction non-degenerate (i.e. has trivial kernel), the desired conclusion follows from the essential uniqueness (\cite[Theorem 2.5]{slw_fact}, \cite[\S 1.1]{el_kproj}) of the irreducible projective representation attached to a cocycle.

  \item \textbf{ Injectivity: \Cref{item:projflc1}} If two projectively flat rank-$q$ vector bundles $\cE_i$ have the same first Chern class then, by \Cref{re:munclass}, the two flat \cite[Proposition I.4.23]{kob_cplx} $q\times q$ matrix bundles $\cE_i\otimes \cE_i^*$ have the same characteristic class. Having disposed of part \Cref{item:flbeta}, this means that
    \begin{equation*}
      \cE_1\otimes \cE_1^*
      \cong
      \cE_2\otimes \cE_2^*.
    \end{equation*}
    It follows \cite[Théorème 9]{dd} that $\cE_2\cong \cE_1\otimes \cL$ for some line bundle $\cL$, and hence
    \begin{equation*}
      c_1(\cE_2) = c_1(\cE_1) + q c_1(\cL)
      \quad(\text{by \cite[\S II.1, (1.9) and (1.10)]{kob_cplx}, for instance}).
    \end{equation*}
    Because we are also assuming $c_1(\cE_i)$ equal and $H^2(\bT^n,\bZ)$ is torsion-free,
    \begin{equation*}
      c_1(\cL)=0
      \xRightarrow{\quad}
      \cL\text{ is trivial}
      \xRightarrow{\quad}
      \cE_1\cong \cE_2. 
    \end{equation*}
  \end{enumerate}
  This completes the proof. 
\end{th:protoriok}

In reference to realizing bundles on a (pro-)torus $\bT$ as a pullback through a quotient pro-torus $\bT\xrightarrowdbl{\pi}\ol{\bT}$, note that triviality along the fibers of $\pi$ is certainly not sufficient:

\begin{example}\label{ex:not.plbk}
  Consider a chain
  \begin{equation*}
    \begin{tikzpicture}[>=stealth,auto,baseline=(current  bounding  box.center)]
      \path[anchor=base] 
      (0,0) node (l) {$\bT^2$}
      +(2,.5) node (u) {$\bT^2$}
      +(4,0) node (r) {$\bT^2$}
      ;
      \draw[->>] (l) to[bend left=6] node[pos=.5,auto] {$\scriptstyle \text{$n_{\ell}$-fold }\pi_{\ell}$} (u);
      \draw[->>] (u) to[bend left=6] node[pos=.5,auto] {$\scriptstyle \text{$n_{r}$-fold }\pi_{r}$} (r);
      \draw[->>] (l) to[bend right=6] node[pos=.5,auto,swap] {$\scriptstyle \pi$} (r);
    \end{tikzpicture}
  \end{equation*}
  of covering morphisms. Naturally, any pullback $\pi_{\ell}^*\cE$ will be trivial on $\pi$-fibers. Were it a $\pi$-pullback, we would have
  \begin{equation}\label{eq:pipie}
    \pi_{\ell}^*\cE
    \cong
    \pi^* \cF
    \cong
    \pi_{\ell}^* \pi_r^* \cF
    \xRightarrow{\quad}
    \cE\cong \pi_r^* \cF,
  \end{equation}
  $\pi^*_{\ell}$ being one-to-one on bundle isomorphism classes (e.g. because bundles are determined by their ranks and $c_1$, per \cite[Proposition, p.2]{thad_introtop}). We can easily arrange for the right-hand side of \Cref{eq:pipie} to fail though: it suffices to select $\cE$ with $n_r$ not dividing $c_1(\cE)\in H^2(\bT^2,\bZ)\cong \bZ$. 
  
  This captures essentially the same phenomenon that drives \cite{73127} (in a somewhat different, algebraic-geometric context). 
\end{example}

\begin{remark}\label{re:w1w2chern}
  There is a common pattern to \cite[Theorems 3.2 and 3.3]{as_vbtor}, \Cref{le:flatistriv} and \Cref{th:protoriok}\Cref{item:flbeta} that it might be worthwhile spelling out. 

  The three results concern flat principal bundles on $\bT^n$ with structure groups $\bG=O(q)$, $U(q)$ and $PU(q)$ respectively, described by \cite[Proposition I.2.6]{kob_cplx} morphisms
  \begin{equation}\label{eq:rhotn2g}
    \bZ^n\cong \pi_1(\bT^n)
    \xrightarrow{\quad\rho\quad}
    \bG.
  \end{equation}
  In each case the statement is to the effect that the isomorphism class of the bundle is determined by the {\it obstruction} \cite[\S 12]{ms_char} to reducing the structure group to the universal cover $\widetilde{\bG_0}\xrightarrow{\ }\bG_0$ of the identity connected component $\bG_0\le \bG$, as we now review.


  \begin{enumerate}[(1),wide]

  \item \cite[Theorems 3.2 and 3.3]{as_vbtor} discuss real vector bundles, i.e. the structure group of interest is $\bG=O(q)$. Its identity component is $\bG_0=SO(q)$, and the universal cover $\widetilde{\bG_0}$ is the {\it Spin group} (\cite[Definition B.3.20]{hn_lgp}, \cite[Definition 3.12]{abs_cliff}) $\mathrm{Spin}(q)$. 

    The first Stiefel-Whitney class $w_1\in H^1(\bT^n,\bZ/2)$ controls \cite[Theorem II.1.2]{lm_spin} the {\it orientability} of the bundle, i.e. it is the obstruction to factoring \Cref{eq:rhotn2g} through $\bG_0=SO(q)$, while (assuming $w_1=0$, i.e. for orientable bundles) the {\it second} Stiefel-Whitney class $w_2\in H^2(\bT^n,\bZ/2)$ is the obstruction \cite[Theorem II.1.7]{lm_spin} to the existence of a {\it spin structure}.
    
  \item For \Cref{le:flatistriv} we have $\bG=U(q)$, which is already connected and incurs no obstruction: flat complex vector bundles have trivial Chern classes \cite[Proposition II.3.1 (a)]{kob_cplx}.

  \item Finally, in \Cref{th:protoriok}\Cref{item:flbeta} the group $\bG=PU(q)$ is again connected, its universal cover is $SU(q)$ \cite[Proposition 17.1.3]{hn_lgp}, and the obstruction is the 2-class $\beta\in H^2(\bT^n,\mu_q)$ discussed in \Cref{re:munclass}.
  \end{enumerate}
\end{remark}

\section{An aside on quantum tori}\label{se:tor}

To connect \Cref{th:protoriok} to the \emph{non-commutative torus algebras} $A_{\theta}^n:=C\left(\bT^n_{\theta}\right)$ (\cite[\S 1]{rief_case}, \cite[pre Proposition 12.8]{gbvf_ncg}) attached to a rational skew-symmetric matrix $\theta\in M_n(\bQ)$, recall the invariant
\begin{equation*}
  \begin{aligned}
    q^2_{\theta}
    &:=
      \text{cardinality}
      \left|
      \im\left(
      \bZ
      \xrightarrow[\quad\forall\left(\ell\in \bZ\ :\ \ell\theta\in M_n(\bZ)\right)\quad]{\quad\ell\theta\quad}
      \bZ
      \xrightarrowdbl{\ }
      \bZ/\ell
      \right)
      \right|
    \\
    &\xlongequal[\quad\text{\cite[Lemma 2.1]{2508.04922v1}}\quad]{\quad}
      \quad
      \text{index}
      \quad
      \left[\left(\bZ^n+\im\theta\right):\bZ^n\right]
  \end{aligned}
\end{equation*}
of \cite[(2-3)]{2508.04922v1} (where $q^2_{\theta}$ is denoted by $h_{\theta}$ instead). The proof of \cite[Lemma 2.1]{2508.04922v1} argues that $q^2_{\theta}$ is a perfect square (so that $q_{\theta}:=\sqrt{q^2_\theta}\in \bZ_{\ge 1}$, as the notation suggests), and by \cite[Proposition 2.6]{2508.04922v1} we have an isomorphism
\begin{equation}\label{eq:tor.mtrx.bdl}  
  C\left(\bT^n_{\theta}\right)
  \cong
  \End\left(\cE_{\theta}\right)
  :=
  \Gamma\left(\cE_{\theta}\otimes \cE_{\theta}^*\right)
  ,\quad
  \cE_{\theta}\text{ rank-$q_{\theta}$ projectively flat on}
  \Max(Z(A^n_{\theta}))\cong \bT^n
\end{equation}
with $\Gamma(\bullet)$ denoting the section space of a (vector/algebra) bundle. 

Consider a torsion-free abelian group $\Gamma$ equipped with a bi-additive skew-symmetric
\begin{equation*}
  \Gamma\times \Gamma
  \xrightarrow{\quad\sigma \quad}
  \bS^1
\end{equation*}
and the resulting (completed) \emph{cocycle twist} (\cite[\S 2.2]{el_morita-def} or \cite[Definition 7.1.1]{mont} for the broader Hopf-algebraic framework) $A_{(\Gamma,\sigma)}$ of the group algebra $\bC \Gamma$. Per \cite[Theorem, p.137]{dekr} the isomorphism class of $A_{(\Gamma,\sigma)}$ determines that of $(\Gamma,\sigma)$ provided $\sigma$ takes only torsion values. \Cref{th:rat.n.tori} affords a bundle-theoretic approach to this (at least for finitely-generated $\Gamma$).


\pf{th:rat.n.tori}
\begin{th:rat.n.tori}
  That $m$ and $n$ can be extracted (for arbitrary \emph{real} skew-symmetric $\theta$) from the isomorphism class of $A^n_{\theta}\otimes M_m$ is observed in \cite[Theorem A(2)]{2508.04922v1}, so it is enough to assume $n=n'$ and $m=m'$ throughout (incidentally simplifying the notation).

  \begin{enumerate}[label={},wide]
  \item\textbf{\Cref{item:th:rat.n.tori:tht} $\Rightarrow$ \Cref{item:th:rat.n.tori:tn}.} Adding integer matrices to $\theta$ will not affect the isomorphism class of $A^n_{\theta}$, for only
    \begin{equation*}
      \exp\left(2\pi i\cdot \left(\text{entries }\theta_{ij}\right)\right)
      \in
      \bS^1
    \end{equation*}
    enter the latter's definition \cite[pp.193-194]{rief_case}. On the other hand, $A^n_{\theta}$ being the cocycle twist of the group algebra $\bC \bZ^n$ by
    \begin{equation*}
      \bZ^n\times \bZ^n
      \xrightarrow{\quad\exp(2\pi i\theta)\quad}
      \bS^1
      \quad
      \left(\text{$\theta$ regarded as a bilinear map}\right),
    \end{equation*}
    $\theta\mapsto T\theta T^t$ for $T\in GL(n,\bZ)\cong\Aut(\bZ^n)$ implements an automorphism of the underlying $\bZ^n$.

  \item\textbf{\Cref{item:th:rat.n.tori:tn} $\Rightarrow$ \Cref{item:th:rat.n.tori:tnm}} is immediate.
    
  \item\textbf{\Cref{item:th:rat.n.tori:tnm} $\Rightarrow$ \Cref{item:th:rat.n.tori:tht}.} Observe that 
    \begin{equation*}
      A^n_{\theta}\otimes M_m
      \ 
      \overset{\text{\Cref{eq:tor.mtrx.bdl}}}{\cong}
      \ 
      \End\left(\cE_{\theta}\right)\otimes M_m
      \ 
      \cong
      \ 
      \End\left(\cE_{\theta}^{\oplus m}\right)
    \end{equation*}
    and similarly for $\theta'$ for projectively flat bundles on the homeomorphic (by \Cref{item:th:rat.n.tori:tnm}) spectra
    \begin{equation*}
      X_{\theta}
      :=
      \Max\left(Z(A^n_{\theta})\right)
      \cong
      \bT^n
      \cong
      \Max\left(Z(A^n_{\theta'})\right)
      =:
      X_{\theta'}.
    \end{equation*}
    Said bundles moreover have equal ranks $q:=q_{\theta}=q_{\theta'}$, for $mq_{\theta,\theta'}$ is (respectively) the common dimension of all irreducible $A^n_{\theta,\theta'}\otimes M_m$-representations.

    Two equal-rank algebra bundles $\End(\cE_{\bullet})$, $\bullet\in \{\theta,\theta'\}$ over a common torus $\bT^n$ will be isomorphic precisely \cite[Théorème 9]{dd} when
    \begin{equation*}
      \exists\left(\text{line bundle }\cL\right)
      \quad:\quad
      \cE_{\theta'}\cong \cE_{\theta}\otimes \cL,
    \end{equation*}
    which under the additional projective flatness condition that obtains in the present case is also equivalent by \Cref{th:protoriok} to $q|c_1(\cE_{\theta})-c_1(\cE_{\theta'})$ (divisibility in the obvious sense, in the torsion-free abelian group $H^2(\bT^n,\bZ)$). The same applies to $\cE^{\oplus m}_{\bullet}$: the condition in that case is
    \begin{equation}\label{eq:mq.div}
      \begin{aligned}
        mq
        |
        c_1\left(\cE^{\oplus m}_{\theta}\right)-c_1\left(\cE^{\oplus m}_{\theta'}\right)
        &\xlongequal[\quad]{\quad\text{\cite[Theorem 4.4.3(II)]{hirz_top}}\quad}
          m\left(c_1(\cE_{\theta})-c_1(\cE_{\theta'})\right)\\
        &\xLeftrightarrow{\quad}
          \quad
          q|c_1(\cE_{\theta})-c_1(\cE_{\theta'}),
      \end{aligned}
    \end{equation}
    i.e. $m$ makes no difference. 

    Recall next (\cite[proofof Lemma 2.1]{2508.04922v1}, via \cite[\S IX.5.1, Th\'eor\`eme 1]{bourb_alg-09}) that up to a transformation $\theta\mapsto T\theta T^t$, $T\in GL(n,\bZ)$ we may assume
    \begin{equation*}
      \theta
      =
      \begin{pmatrix}
        \phantom{-}0 & D & 0\\
        -D & 0 & 0\\
        \phantom{-}0&0&0
      \end{pmatrix}
      \quad
      \left(\text{square-block entries}\right)
      ,\quad
      D=\mathrm{diag}\left(\text{lowest-terms }\frac{p_i}{q_i}\in \bQ\right)_i,
    \end{equation*}
    so that $\cE_{\theta}$ decomposes as an exterior tensor product of trivial rank-1 bundles on circles and rank-$q_i$ bundles $\cE_{\theta,i}$ on 2-tori. The rank $q$ is $\prod_i q_i$, and \Cref{le:rief} applied to the individual $i$-indexed $\bT^2$ factors makes it clear that (perhaps up to a sign)
    \begin{equation*}
      q\theta=
      c_1\left(\cE_{\theta}\right)
      \in
      H^2(\bT^n,\bZ)
      \overset{\text{\Cref{eq:cohalt}}}{\cong}
      \Hom\left(\bZ\wedge \bZ,\bZ\right)
      \quad
      \left(\text{skew-symmetric bi-additive maps maps}\right). 
    \end{equation*}
    This and its $\theta'$ counterpart plainly render \Cref{eq:mq.div} equivalent to the rightmost condition in \Cref{item:th:rat.n.tori:tht}. 
  \end{enumerate}
\end{th:rat.n.tori}


\addcontentsline{toc}{section}{References}

\def\polhk#1{\setbox0=\hbox{#1}{\ooalign{\hidewidth
  \lower1.5ex\hbox{`}\hidewidth\crcr\unhbox0}}}
  \def\polhk#1{\setbox0=\hbox{#1}{\ooalign{\hidewidth
  \lower1.5ex\hbox{`}\hidewidth\crcr\unhbox0}}}
  \def\polhk#1{\setbox0=\hbox{#1}{\ooalign{\hidewidth
  \lower1.5ex\hbox{`}\hidewidth\crcr\unhbox0}}}
  \def\polhk#1{\setbox0=\hbox{#1}{\ooalign{\hidewidth
  \lower1.5ex\hbox{`}\hidewidth\crcr\unhbox0}}}
  \def\polhk#1{\setbox0=\hbox{#1}{\ooalign{\hidewidth
  \lower1.5ex\hbox{`}\hidewidth\crcr\unhbox0}}}
  \def\polhk#1{\setbox0=\hbox{#1}{\ooalign{\hidewidth
  \lower1.5ex\hbox{`}\hidewidth\crcr\unhbox0}}}

\Addresses

\end{document}